\documentclass[a4paper]{article}
\usepackage{amsmath}
\usepackage{amsfonts}
\usepackage{amssymb}
\usepackage{amsthm}
\usepackage{mathtools}
\usepackage{tikz}
\usepackage{tikz-cd}
\usepackage{enumerate}
\usepackage{layout}
\usepackage{mathabx}
\usepackage{bm}
\usetikzlibrary{patterns}

\DeclareMathOperator{\id}{id}

\DeclareMathOperator{\cd}{cd}

\DeclareMathOperator{\Hom}{Hom}

\DeclareMathOperator{\Ext}{Ext}

\DeclareMathOperator{\im}{im}

\newcommand{\iso}{\ensuremath{\cong}}
\newcommand{\Z}[1][]{\ensuremath{\mathbb{Z}_{#1}}}

\newcommand{\N}{\ensuremath{\mathbb{N}}}

\newcommand{\F}{\ensuremath{\mathbb{F}}}

\newcommand{\nsgp}[1][]{\ensuremath{\triangleleft_{#1}}}

\newcommand{\gp}[1]{\ensuremath{\langle #1\rangle}}

\newcommand{\famS}[0]{\ensuremath{\mathcal{S}}}
\newcommand{\famT}[0]{\ensuremath{\mathcal{T}}}

\newcommand{\Zpiof}[1]{{\ensuremath{\Z[{\pi}][\![#1]\!]}}}

\newcommand{\ZGS}{\ensuremath{\Z[\pi][\![G/\famS]\!]}}

\newcommand{\lqt}{\backslash}

\newtheorem{theorem}{Theorem}[section]

\newtheorem{prop}[theorem]{Proposition}

\newtheorem{lem}[theorem]{Lemma}
\newtheorem{clly}[theorem]{Corollary}

\theoremstyle{definition}
\newtheorem{defn}[theorem]{Definition}
\newtheorem*{cnv}{Conventions}

\theoremstyle{remark}
\newtheorem*{rmk}{Remark}
\newtheorem{example}[theorem]{Example}
\theoremstyle{plain}

\newcounter{introthmcount}
\newenvironment{introthm}[1]{\\[1.9ex] {\bf Theorem \ref{#1}}. \em}{\em \vspace{2ex}}
\theoremstyle{definition}

\DeclareMathOperator{\inc}{in}
\DeclareMathOperator{\ev}{ev}

\newcommand{\cC}[0]{\ensuremath{\mathcal{C}}}
\newcommand{\cD}[0]{\ensuremath{\mathcal{D}}}

\numberwithin{equation}{section}

\title{Relative extensions and cohomology of profinite groups}
\author{Gareth Wilkes}
\begin{document}
\maketitle

\begin{abstract}
We construct a correspondence between the cohomology groups of a group $G$ relative to a family of subgroups $\famS$ and the classes of `relative extensions' of $G$ by abelian groups, modulo a certain equivalence relation. We establish this correspondence for both discrete and profinite group pairs. We go on to discuss the relationships of profinite group pairs of cohomological dimension one with free products and projective group pairs.
\end{abstract}

\section{Introduction}

One of the most important applications of group cohomology that a new student will meet, and the example often used to convince such a student that $H^2(G)$ actually `means' something, is the use of the second cohomology group $H^2(G;A)$ to classify equivalence classes of extensions of $G$ by $A$. With some mild modifications to acknowledge the topology, a similar correspondence holds for extensions of profinite groups by a finite abelian group \cite[Section 6.8]{RZ00}. 

The relationship of $H^2$ to extension classes is especially important for profinite groups. Here it leads immediately to a characterisation of projective profinite groups as groups of cohomological dimension one \cite[Sections 7.5--7.6]{RZ00}. For pro-$p$ groups this yields a proof that groups of cohomological dimension one are free, a parallel to the famous Stallings-Swan theorem on discrete groups of dimension one \cite{Stallings68, Swan69}.

Rather less well-known than the standard cohomology groups are the relative cohomology groups $H^n(G,\famS)$, where $\famS = \{S_i\}_{i\in I}$ is some family of (not necessarily distinct) subgroups of $G$. Such cohomology groups have been defined both for discrete groups and, more recently, for profinite groups (under some restrictions on the family of subgroups). In this paper we will study the connection of the relative cohomology group $H^2(G,\famS; A)$ with some notion of an extension of $G$ `relative to \famS'.

We include first a full discussion for discrete groups. It seems possible that the analysis for discrete groups is already known, but the author was unable to find any record in the literature of a relationship between $H^2(G,\famS;A)$ and extensions.   

There are a few different notions of `extension relative to a collection of subgroups' one could conceive, and a few different equivalence relations one could choose to impose. The `correct' definition is given in Definition \ref{def:RelExt}. For this definition, a {\em relative extension} of $G$ by $A$ is an extension $E$ of $G$ by $A$ together with a choice of splitting homomorphism $S\to E$ for each peripheral subgroup $S$. 

That this notion is a reasonable one may be immediately seen from the long exact sequence of relative homology, of which a part is:
\[\prod_{i\in I} H^1(S_i;A) \to H^2(G,\famS; A) \to H^2(G;A) \to \prod_{i\in I} H^2(S_i; A).\]
We see that some portion of $H^2(G,\famS; A)$ comes from the kernel of the map $H^2(G;A) \to \prod H^2(S_i; A)$---that is, extensions of $G$ by $A$ which are trivial over each $S_i$. The rest of $H^2(G,\famS; A)$ arises from the groups $H^1(S_i;A)$, which parametrize the different splittings of the groups $A\rtimes S_i$. 

Having made the right definitions, we turn the data of such a relative extension into a certain cochain specifying an element of $H^2(G,\famS;A)$, and vice versa. This will establish the following theorem.
\begin{introthm}{thm:discretecase}
Let $(G,\famS)$ be a group pair and let $A$ be a $G$-module. There is a natural bijection between $H^2(G,\famS;A)$ and the equivalence classes of relative extensions of $(G,\famS)$ by $A$.
\end{introthm}

Having proved this theorem for discrete groups, we turn our attention to ground which is certainly unbroken: profinite groups. One would hope a similar theorem holds---provided one restricts to extensions with a finite kernel $A$, a standard concession in this world. The modifications to account for topology are no longer `mild': there are substantive obstructions to carrying out the proof verbatim, as discussed in Section \ref{sec:ProfiniteObstructions}. In particular, the necessity of ensuring all our maps are continuous at every stage introduces new complexities into the proof. We resolve these difficulties in Section \ref{sec:ProfiniteRelExts}, and produce the hoped-for theorem.
\begin{introthm}{thm:profinitecase}
Let $(G,\famS)$ be a profinite group pair and let $A$ be a finite $G$-module. There is a natural bijection between $H^2(G,\famS;A)$ and the equivalence classes of relative extensions of $(G,\famS)$ by $A$.
\end{introthm}

To round out the paper, we connect our relative cohomology groups to the study of projective pairs. This material is closely related to work of Ribes \cite{Ribes74} and Haran \cite{Haran87} and to new independent work released shortly before this paper by Haran--Zalesskii \cite{HZ22}. The author also notes that a homological connection will be explored by Zalesskii in a paper currently in preparation \cite{Zal22up}. The author is grateful to Prof.~Zalesskii for forewarning of this paper, and for other useful discussions.

Each of these investigations has as a goal some theorem such as the following.
\begin{introthm}{thm:cd1isproj}
Let $(G,\famS)$ be a pro-$p$ group pair, where $\famS=\{S_x\}_{x\in X}$ is a family of subgroups continuously indexed over the profinite set $X$. The following are equivalent.
\begin{enumerate}[(I)]
\item $\cd_p(G,\famS)\leq 1$.
\item $(G,\famS)$ is a projective profinite pair.
\item $G= \coprod_{x\in X} S_x \amalg F$ is a free profinite product of the family \famS\ with a free pro-$p$ group $F$.
\end{enumerate}\end{introthm}
The new part of this theorem is the involvement of relative cohomology, which provides a rather more concise and natural condition than the more complicated statements of \cite{Ribes74}. An analogous result for discrete groups, under the imposition of extra conditions, was proved by Wall \cite{Wall71}.

The outline of the paper is as follows. In Section \ref{sec:Background} we will recall some definitions and  notation pertaining to profinite modules and relative cohomology. Section \ref{sec:ChainHomotopies} will set up the homological algebra framework: chain maps, homotopies, and homotopies-between-homotopies. 

In Section \ref{sec:RelExtsCohom} we will establish the equivalence, for discrete groups, of second relative cohomology and certain equivalence classes of extensions. Subsequently, Section \ref{sec:TheProfiniteCase} will describe the adaptations necessary to perform this programme for profinite group pairs.

Finally, Section \ref{sec:Projective} will discuss some connections and applications to projective profinite pairs and the Kurosh Subgroup Theorem.

\begin{cnv}
There is a considerable amount of notation in this paper, and particularly in the middle sections our alphabets are somewhat strained for room. We will attempt to keep things in order with the following guidelines.
\begin{itemize}
\item $A$ will be an abelian group and $G$-module.
\item $C_\bullet$, $D_\bullet$ and variants thereof will denote chain complexes. The boundary maps in these chain complexes will be $d_\bullet$, without a decoration to distinguish which chain complex they belong to. Cochain complexes and coboundary maps will be denoted with $C^\bullet$, $D^\bullet$, $d^\bullet$ etc.
\item $E,F,G,H$ will denote groups. $\famS$ will be a family of subgroups of $G$.
\item Especially in Sections \ref{sec:RelExtsCohom} and \ref{sec:TheProfiniteCase}, the letters $I_\bullet$, $J_\bullet$, $L_\bullet$, $M_\bullet$, $N_\bullet$ will denote various chain maps and chain homotopies which will play the roles specified in Section \ref{sec:ChainHomotopies}.
\item Lower case Greek letters will be used in Sections \ref{sec:RelExtsCohom} and \ref{sec:TheProfiniteCase} for various cochains. Lower case Roman letters $p,q,r$ will be used here for functions between groups.

Once the cochains have disappeared from the discussion in Section \ref{sec:Projective}, we will repurpose the Greek letters for use as maps between groups.
\item Dual maps will be denoted with a star, e.g.\ $I_\bullet^\ast$.
\item For $p$ a prime, $\F_p$ denotes a finite field with $p$ elements. The ring of $p$-adic integers will be denoted \Z[p]. If $\pi$ is a set of primes then \Z[\pi] denotes the ring of $\pi$-adic integers $\Z[\pi] = \prod_{p\in \pi}\Z[p]$.
\end{itemize}
\end{cnv}

\section{Background}\label{sec:Background}

In this section we recall certain background concepts, mainly for the purpose of setting up notation.

\subsection{Profinite modules}

Given a profinite set $X$ and a profinite ring $R$ we denote  by $R[\![X]\!]$ the free profinite $R$-module on the set $X$, defined by the familiar universal property: maps out of the free module are uniquely determined by a continuous function on $X$. See \cite[Section 5.2]{RZ00}. If $X$ is finite then this agrees with the abstract free module $R[X]$. 

In particular, if $G$ is a profinite group, we may form the completed group ring \Zpiof{G} and the free profinite $G$-module $\Zpiof{G}[\![X]\!]$, where $\pi$ is some non-empty set of primes. See \cite[Sections 5.3--5.7]{RZ00}. When $X$ admits a continuous $G$-action, the abelian group \Zpiof{X} becomes a \Zpiof{G}-module in a natural way.

Given a finite collection of profinite $G$-modules $\{M_i\}$, the direct sum $\bigoplus_i M_i$ is also a profinite $G$-module. However, when the family is infinite, the direct sum will no longer be compact and we must use a profinite direct sum. Background material on this operation may be found in \cite[Section 9.1]{Ribes17}\footnote{This section actually refers the reader back to \cite[Section 5.1]{Ribes17}, where the very similar theory of free profinite products is established. Hence the inclusion of a second reference to specifically deal with sums of modules.} and \cite[Appendix A]{WilkesRelCoh}. We include an abbreviated definition here, but for the purposes of this paper it will suffice to keep in mind that most properties of profinite direct sums mimic abstract direct sums.

Let $\mu\colon {\cal M}\to X$ be a surjection of profinite spaces, such that every $M_x = \mu^{-1}(x)$ is a profinite $R$-module. Assume that the operations of these $R$-modules are defined continuously\footnote{Precise definition: \cite[Definition A.1]{WilkesRelCoh}.} with respect to $x$. Such an object may be called a {\em sheaf of $R$-modules}. 

Let $\cal U$ be the set of all finite index submodules $U$ of $A= \bigoplus_{x\in X} M_x$ such that the composition ${\cal M}\to A \to  A/U$ is continuous. We then define the profinite direct sum of the modules $M_x$ to be the inverse limit
\[\bigboxplus_{x\in X} M_x = \varprojlim_{U\in \cal U} A /U. \]
Observe that this object, a completion of the abstract direct sum, comes with a natural continuous function ${\cal M} \to \bigboxplus_x M_x$. The profinite sum then satisfies (\cite[Proposition A.3]{WilkesRelCoh}) the natural universal property: for any profinite $R$-module $N$, morphisms $\bigboxplus_x M_x \to N$ are uniquely determined by a continuous function ${\cal M}\to N$ which restricts to a morphism of $R$-modules $M_x\to N$ for each $x$. When $X$ is finite, this object coincides with the usual direct sum $\bigoplus_x M_x$.

We note two properties in particular that will be of use in this paper. Suppose that $\alpha\colon Z \to X$ is a surjection of profinite spaces, and set $Z_x = \alpha^{-1}(x)$ for $x\in X$. Then the free modules $R[\![Z_x]\!]$ form a sheaf over $X$ in a natural way, and we have an isomorphism (cf \cite[Proposition A.14]{WilkesRelCoh})
\begin{equation} R[\![Z]\!] = \bigboxplus_{x\in X} R[\![Z_x]\!].\label{eq:profdirectsum}
\end{equation}
Suppose now we have a commuting diagram of profinite spaces
\[\begin{tikzcd} Z \ar{r} \ar[two heads]{d} & W \ar[two heads]{d} \\ 
X \ar{r} & Y \end{tikzcd}\]
where $Y$ is finite. From the above relation and the universal properties, we obtain a natural morphism
\begin{equation}
\begin{tikzcd} \bigboxplus_{x\in X} R[\![Z_x]\!] \ar{r}& \bigoplus_{y\in Y} R[\![W_y]\!].\end{tikzcd}\label{eq:MapsOfDirectSums}
\end{equation}

\subsection{Relative cohomology: discrete groups}

Relative homology and cohomology theory for discrete groups was developed by various authors including Auslander \cite{auslander1954relative}, Ribes \cite{Ribes69} and Takasu \cite{takasu1959relative}. As a good reference we recommend Bieri and Eckmann \cite{BE77}, and it is their notation that we will follow. We will not need extensive knowledge of results, but we will recall the notation and definitions.

Let $G$ be a group and let $\famS = \{S_i\}_{i\in I}$ be a family of (not necessarily distinct) subgroups of $G$ indexed by a set $I$. We denote by $\Z{}[G/\famS]$ the $G$-module
\[\Z{}[G/\famS] = \bigoplus_{i\in I} \Z{} [G/S_i].\]
We have the standard `augmentation' or `evaluation' map \[\ev\colon \Z{}[G/\famS]\to \Z \]
and denote the kernel of this map by $\Delta$ or $\Delta_{G/\famS}$. Calling the inclusion map $\inc$, we acquire an important short exact sequence
\begin{equation}
\begin{tikzcd}0 \ar{r}& \Delta_{G/\famS}\ar{r}{\inc} & \Z{}[G/\famS] \ar{r}{\ev} & \Z \ar{r}& 0. \end{tikzcd}\label{eq:DefnOfDeltaDiscrete}
\end{equation}
Observe that this sequence is split as a sequence of \Z-modules, and thus remains exact when a functor $\Hom(-,A)$ is applied\footnote{Note that this is the group of all homomorphisms, not just the $G$-linear homomorphisms $\Hom_G$}.

For a $G$-module $A$, we now define
\[H^n(G,\famS; A) = H^{n-1}(G; \Hom(\Delta, A)) = \Ext^{n-1}_G(\Z, \Hom(\Delta,A))\]
where $G$ acts diagonally on the homomorphism group $\Hom(\Delta,A)$. Note the shift in dimensions. 

\subsection{Relative cohomology: profinite groups}

Let us now more to the relative (co)homology of profinite group pairs. For this theory we refer to the author's paper \cite{WilkesRelCoh}, which itself builds on cohomology theories expounded in Serre \cite{Serre13}, Ribes--Zalesskii \cite{RZ00} and Symonds--Weigel \cite{SW00}. Related to cohomology theory, a notion of a Poincar\'e duality pair of profinite groups was defined by Kochloukova in \cite{kochloukova2014pro}.

The notation here mimics the discrete case, but the use of the symbol $\famS$ is subtly different. For discrete groups, \famS\ was simply a set of subgroups of $G$ indexed over some set $I$. For profinite groups, we have slightly more structure.

\begin{defn}\label{DefCtsIndex}
Let $G$ be a profinite group, and $\famS=\{S_x\}_{x\in X}$ a family of subgroups of $G$ indexed by a profinite space $X$. We say that \famS{} is {\em continuously indexed by $X$} if whenever $U$ is an open subset of $G$ the set \[\left\{x\in X \mid S_x\subseteq U\right\}\] is open in $X$. An equivalent definition is that the set
\[\famS = \left\{(g,x)\in G\times X\mid g\in S_x \right\} \]
is a closed subset of $G\times X$. 
\end{defn}

Note that we will use the symbol $\famS$ to denote the last set---that is, the subgroups $S_x$ `gathered together' over $X$---as well as to denote the abstract family $\famS=\{S_x\}_{x\in X}$. This slight conflation ought not to cause any confusion.

Let $G$ be a pro-$\cal C$ group for some variety $\cal C$ of finite groups, and let $\pi$ be some non-empty set of primes. Let $\famS$ be a family of subgroups of $G$ continuously indexed by $X$. The set 
\[G/\famS = \bigsqcup_{x\in X} G/S_x\]
acquires a natural topology as a quotient of the set $G\times X$, and the natural action of $G$ on this set is continuous \cite[Proposition 1.38]{WilkesRelCoh}. We may thus form the $G$-module \ZGS\ as the free \Z[\pi]-module on this space. We have an expression 
\[\ZGS = \bigboxplus_{x\in X} \Zpiof{G/S_x}\]
of \ZGS\ as a profinite direct sum \eqref{eq:profdirectsum}. 

We now have a short exact sequence to mimic Equation \eqref{eq:DefnOfDeltaDiscrete}:
\begin{equation}
\begin{tikzcd}0 \ar{r}& \Delta_{G/\famS}\ar{r}{\inc} & \ZGS \ar{r}{\ev} & \Z[\pi] \ar{r}& 0 \end{tikzcd}\label{eq:DefnOfDeltaProfinite}
\end{equation}
and will use the same notation $\Delta=\Delta_{G/\famS}$ for the kernel.

Take now a discrete $\pi$-primary $G$-module $A$. The space of continuous homomorphisms from a pro-$\pi$ abelian group to $A$ is itself a discrete $\pi$-primary $G$-module, in the domain of definition of profinite cohomology theory. We are thus permitted to make the definition
\[H^n(G,\famS; A) = H^{n-1}(G; \Hom(\Delta, A)) = \Ext^{n-1}_G(\Z[\pi], \Hom(\Delta,A))\]
for such a $G$-module $A$. Here, as in all cases involving a profinite group, $\Hom$-groups are deemed to be groups of continuous homomorphisms.

One last piece of notation we will use in this paper is the cohomological dimension. We will use the following as a definition, although strictly speaking it is a proposition.
\begin{defn}
Let $p\in \pi$. We write $\cd_p(G,\famS)\leq n$ if $H^{n+1}(G,\famS;A)=0$ for all $p$-primary modules $A$.
\end{defn}

\section{Choices of chain homotopies}\label{sec:ChainHomotopies}

Any student of homological algebra has seen the procedure of comparing projective resolutions with the use of chain maps, whose compositions are homotopic to the identity via a chain homotopy. In the circumstances of this paper, we will need to not only have the existence of such maps, but to verify that certain quantities are independent of {\em choice} of chain homotopy. We must therefore invoke the existence of `homotopies of homotopies', which one might term 2-homotopies.

Let $(\cC_\bullet, d_\bullet)$ and $(\cD_\bullet, d_\bullet)$ be two resolutions of the same object $Z$ in an abelian category. Assume that $\cC_\bullet$ is projective. Let $I_\bullet\colon \cD_\bullet \to \cC_\bullet$ be a chain map lifting $\id_Z$, and let $J_\bullet, \widetilde J_\bullet\colon \cC_\bullet \to \cD_\bullet$ be two choices of chain maps lifting $\id_Z$.

The difference between the chain maps $J_\bullet$ and $\widetilde J_\bullet$ is mediated by a chain homotopy $M_\bullet\colon \cC_\bullet \to \cD_{\bullet+1}$; that is, we have
\begin{equation} \widetilde J_n - J_n = M_{n-1}d_n + d_{n+1}M_n  \label{eq:DefnOfM}
\end{equation}
with the usual convention that $M_{-1}=0$. These are built in the usual inductive manner: if $M_0,\ldots, M_n$ have been constructed, then we note 
\begin{eqnarray*}
d_{n+1}(\widetilde J_{n+1} - J_{n+1} - M_nd_{n+1}) &=& (\widetilde J_n - J_n - d_{n+1}M_n) d_{n+1} \\
&=& M_{n-1}d_nd_{n+1} = 0, \end{eqnarray*}
and projectivity of $\cC_{n+1}$ gives a map $M_{n+1}$ with the required property.

\[\begin{tikzcd}
{} \ar{r} & \cC_{n+1} \ar{rr}{d_{n+1}} \ar{d}[swap]{J_{n+1}} \ar[shift left = 1ex]{d}{\widetilde J_{n+1}} & &\cC_n \ar{r}\ar[shift left = 1ex]{d}{\widetilde J_{n}} \ar{d}[swap]{J_{n}} \ar[dotted,red]{dll}[red]{M_n}& {}\\
{} \ar{r} & \cD_{n+1} \ar{rr}[swap]{d_{n+1}} && \cD_n \ar{r} & {}
\end{tikzcd}\]

Now, $I_\bullet J_\bullet$ and $I_\bullet \widetilde J_\bullet$ are both chain maps lifting $\id_Z$. We may take choices of chain homotopy $L_\bullet, \widetilde L_\bullet\colon \cC_\bullet \to \cC_{\bullet + 1}$ from $I_\bullet J_\bullet$ and $I_\bullet \widetilde J_\bullet$ to the identity, i.e.
\begin{equation} I_n J_n - \id = d_{n+1} L_n + L_{n-1} d_n  \label{eq:DefnOfL}
\end{equation} 
(and similar equations for $\widetilde J_\bullet$) for each $n$, where again $L_{-1}=0$. 

We claim that the difference between $L_\bullet$ and $\widetilde L_\bullet$ is mediated by a 2-homotopy $N_\bullet\colon \cC_\bullet \to \cC_{\bullet + 2}$ with the following property:
\[\widetilde L_n -L_n = I_{n+1} M_n + d_{n+2} N_n - N_{n-1} d_n\] 
for all $n$. Firstly, we have 
\begin{eqnarray*}
d_1(\widetilde L_0 - L_0) &=& (I_0 \widetilde J_0 -\id) - (I_0J_0 -\id)\\
&=& I_0d_1M_0 \\ &=& d_1I_1M_0.
\end{eqnarray*}
Therefore there exists $N_0\colon \cC_0\to \cC_2$ such that \[\widetilde L_0 - L_0 = I_1M_0 + d_2N_0.\]
Subsequently, suppose we have built $N_0,\ldots, N_n$ with the required property. Then 
\begin{eqnarray*}d_{n+2}\left( \widetilde L_{n+1} - L_{n+1} \right) &=& I_{n+1}\widetilde J_{n+1} - I_{n+1}J_{n+1} - (\widetilde L_n - L_n)d_{n+1}\\
&=& I_{n+1}\left(M_nd_{n+1} + d_{n+2} M_{n+1}\right) - (\widetilde L_n - L_n)d_{n+1}\\
&=& I_{n+1}d_{n+2}M_{n+1} - \left(d_{n+2}N_n + N_{n-1}d_n\right)d_{n+1}\\
&=& d_{n+2}\left(I_{n+2} M_{n+1} - N_nd_{n+1}\right)
\end{eqnarray*}
so by projectivity of $\cC_{n+1}$ there does indeed exist $N_{n+1}\colon \cC_{n+1}\to \cC_{n+3}$ such that 
\begin{equation} \widetilde L_{n+1} - L_{n+1} = I_{n+2} M_{n+1} - N_nd_{n+1} + d_{n+3}N_{n+1}\label{eq:DefOfN}
\end{equation} 
as required.

\[\begin{tikzcd}
{} \ar{r} & \cC_{n+2} \ar{rr}[]{d_{n+2}} \ar[gray]{dd}{} \ar[gray,shift left = 1ex]{dd}{} & & \cC_{n+1} \ar{rr}{d_{n+1}}\ar[gray]{dd}{} \ar[gray,shift left = 1ex]{dd}{} \ar[gray]{ddll}{} \ar[gray,shift left = 1ex]{ddll}{}  & &\cC_n \ar{r}\ar[shift left = 1ex]{dd}{I_n\widetilde J_{n}} \ar{dd}[swap]{I_nJ_{n}} \ar[shift left = 1ex]{ddll}[]{\widetilde L_n} \ar[]{ddll}[swap]{L_n} \ar[red]{ddllll}[swap]{N_n} & {}\\
&&&&&&\\
{} \ar{r} & \cC_{n+2} \ar{rr}[swap]{d_{n+2}} & {}& \cC_{n+1} \ar{rr}[swap]{d_{n+1}} &{} & \cC_n \ar{r} & {}
\end{tikzcd}\]

\section{Relative extensions and cohomology}\label{sec:RelExtsCohom}

In this section we shall set up the relationship between $H^2(G,\famS; A)$ and certain equivalence classes of extensions of $G$ by $A$ for a discrete group $G$.

\begin{defn}\label{def:RelExt}
Let $(G,\famS)$ be a group pair and let $A$ be a $G$-module. A {\em relative extension} of $(G,\famS)$ by $A$ is a pair $(E, (q_i))$ where \[0\to A\to E \stackrel{p}{\to} G\to 1\] is an extension of $G$ by $A$ and $(q_i)$ is a family of homomorphisms $q_i\colon S_i\to E$ for $i\in I$ such that $pq_i = \id_{S_i}$.

Two such relative extensions $(E,(q_i))$ and $(\widetilde E, (\tilde q_i))$ are declared to be {\em equivalent} if there is an equivalence of extensions $f\colon E \to \widetilde E$ such that for each $i$ the maps $fq_i$ and $\tilde q_i$ are {\em $A$-conjugate}: there exists $a_i\in A$ such that 
\[\tilde q_i(s)=a_i\cdot f\circ q_i(s)\cdot a_i^{-1}  \]
for all $s\in S_i$.

The {\em trivial relative extension} is the extension $E=A\rtimes G$ with $q_i(s) = (1,s)$ for all $s\in S_i$.
\end{defn}

Let us set up notation for the various cochain groups required. Let \[\cC_n = C_n(G) = \Z G[G^{(n)}]\] denote the free $\Z G$-module on the set of symbols $[g_1,\ldots, g_n]$, as in the standard resolution of $\Z$ by $\Z G$-modules. We consider $\cC_0$ to be spanned by the empty symbol $[]$. 

Construct the following cochain groups:
\begin{eqnarray*}
C^n(G; A) &=& \Hom_G(\cC_n, A)\\
C^n(\famS; A) &=& \Hom_G(\cC_n, \Hom(\Z{} [G/\famS], A))\\
C^{n+1}(G,\famS; A)&=& \Hom_G(\cC_n, \Hom(\Delta_{G/\famS}, A))
\end{eqnarray*} 
We then have the following commutative diagram with exact rows. Appearing in red are symbols denoting the various cochains we will soon construct to allow the reader to place them easily.
\[\begin{tikzcd}
0 &\ar{l} C^3(G,\famS; A)& \ar{l}[swap]{\inc^\ast} C^2(\famS; A) & \ar{l}[swap]{\ev^\ast} C^2(G;A)& \ar{l} 0\ar[phantom, shift left = 2.5ex, end anchor={[xshift=-3ex]}]{l}[red, at end]{\zeta}\\
0 &\ar{l} C^2(G,\famS; A) \ar{u}{d^2}  & \ar{l}[swap]{\inc^\ast} C^1(\famS; A)\ar{u}{d^2} \ar[phantom, shift left = 2.5ex, end anchor={[xshift=-3ex]}]{l}[red, at end]{\xi}& \ar{l}[swap]{\ev^\ast} C^1(G;A)\ar{u}{d^2} \ar[phantom, shift left = 2.5ex, end anchor={[xshift=-3ex]}]{l}[red, at end]{\eta}& \ar{l} \ar[phantom, shift left = 2.5ex, end anchor={[xshift=-3ex]}]{l}[red, at end]{\upsilon}0\\
0 &\ar{l} C^1(G,\famS; A)\ar{u}{d^1} & \ar{l}[swap]{\inc^\ast} C^0(\famS; A) \ar{u}{d^1}\ar[phantom, shift left = 2.5ex, end anchor={[xshift=-3ex]}]{l}[red, at end]{\omega}& \ar[phantom, shift left = 2.5ex, end anchor={[xshift=-3ex]}]{l}[red, at end]{\chi}&
\end{tikzcd}\]

\subsection{Extensions from cochains}\label{sec:ExtFromCochains}

Take a cocycle $\xi \in C^2(G,\famS; A)$ representing some cohomology class $[\xi]$ in $H^2(G,\famS; A)$. Choose $\eta\in C^1(\famS;A)$ such that $\inc^* \eta = \xi$. We see that $\inc^* d^2\eta =0$, so there exists a unique $\zeta\in C^2(G;A)$ such that $\ev^*\zeta= d^2\eta$. This $\zeta$ defines an extension $E=E_\zeta$ of $G$ by $A$ in the usual manner, viz.\ the set $A\times G$ with multiplication
\[(a_1,g_1)\star (a_2, g_2) = (a_1+ g_1\cdot a_2 + \zeta([g_1,g_2]) , g_1 g_2).\]
So far, this is just the extension of $G$ by $A$ corresponding to the cohomology class $\delta([\xi])\in H^2(G;A)$ produced by the Snake Lemma. To produce a relative extension class, we must also construct the sections $q_i\colon S_i\to E$. These arise from the cochain $\eta\in \Hom_G(\cC_1, \Hom(\Z{} [G/\famS], A))$ by defining
\[q_i(s) = (-\eta([s])(1S_i),s).\]
These maps are indeed homomorphisms. We have:
\begin{eqnarray*}
q_i(s_1)q_i(s_2) &=& (-\eta([s_1])(1S_i)-s_1\cdot \eta([s_2])(1S_i)+\zeta([s_1,s_2]),s_1s_2)\\
&=&(-\eta([s_1])(1S_i)- \eta(s_1[s_2])(s_1^{-1}S_i)+\ev^*\zeta([s_1,s_2])(1S_i),s_1s_2)\\
&=& (-\eta([s_1s_2])(1S_i), s_1s_2)\\ &=& q_i(s_1s_2)
\end{eqnarray*}
using the relation $d^2\eta=\ev^\ast \zeta$. 

From $\xi$ we have thus produced a relative extension $(E,(q_i))$ of $(G,\famS)$ by $A$. Let us show that this extension is well-defined up to equivalence.

Replace $\xi$ by another representative $\tilde\xi=\xi+ d^1\omega$ of the same cohomology class, and make a choice $\tilde\eta$ such that $\inc^*\tilde\eta=\tilde\xi$. Let $\tilde\zeta$ be the resulting cocycle with $\ev^*\tilde\zeta = d^2\tilde\eta$.

Choose some $\chi\in C^0(\famS;A)$ such that $\omega = \inc^*\chi$. Then $\inc^*(\tilde\eta-\eta-d^1\chi)=0$, so there exists $\upsilon \in C^1(G;A)$ such that \[\ev^*\upsilon = \tilde\eta-\eta-d^1\chi.\]
The cocycle $\tilde\zeta$ now satisfies $\tilde \zeta=\zeta + d^2\upsilon$, so we have in the usual way an equivalence of extensions $f\colon E_\zeta\to E_{\tilde\zeta}$ defined by
\[f(a,g)=(a-\upsilon([g]),g).\]

We may now compute the sections $\tilde q_i\colon S_i\to E_{\tilde \zeta}$ arising from $\tilde \eta$:
\begin{eqnarray*}
\tilde q_i(s)&=& (-\tilde\eta([s])(1_iS), s)\\
&=& (-\eta([s])(1S_i) - d^1\chi([s])(1S_i) - \ev^*\upsilon([s])(1S_i), s)\\
&=& (\chi([])(1S_i)-\eta([s](1S_i)-\upsilon([s]) -s\cdot \chi([])(1S_i),s)\\
&=& (a_i,1)\star fq_i(s) \star (-a_i,1)
\end{eqnarray*}
where $a_i=\chi([])(1S_i)$. This provides the required $A$-conjugacy of sections, so the relative extension classes $(E_\zeta,(q_i))$ and $(E_{\tilde\zeta}, (\tilde q_i))$ are equivalent as required.

\subsection{Cochains from extensions}\label{sec:CochainsFromExts}

The more involved direction of the equivalence is to take a relative extension and produce a cocycle $\xi\in C^2(G,\famS;A)$ representing it. 

We will need to recast the chain groups $C^n(\famS; A)$ more concretely in terms of the $S_i\in\famS$. For each $S_i$ we have isomorphisms
\[\Phi\colon \Hom_G(\cC_n, \Hom(\Z{}[G/S_i],A))\longleftrightarrow \Hom_{S_i}(\cC_n,A)\,:\!\Psi\]
given by
\[\Phi(\phi)(c_n)= \phi(c_n)(1S_i),\qquad \Psi(\psi)(c_n)(gS)=g\psi(g^{-1}c_n)\]
for $c_n\in \cC_n$.

Since $\cC_n$ is also a free resolution of \Z\ by $\Z S_i$-modules, we may compare it with the standard resolution $C_n(S_i)=\Z S_i[S_i^{(n)}]$. We have a canonical inclusion of chain complexes $I_n\colon C_n(S) \to \cC_n$. We are thus in the situation of Section \ref{sec:ChainHomotopies}, and may compare the chain complexes $\cC_n$ and $C_n(S_i)$ using maps $J_\bullet$, $L_\bullet$ etc.\ as in that section. 

Using the definition $\Z {}[G/\famS] = \bigoplus_{i\in I} \Z{}[G/S_i]$, we may assemble these various complexes into products. We thus have isomorphisms
\[\Phi\colon \Hom_G(\cC_n, \Hom(\Z{}[G/\famS],A))\longleftrightarrow\prod_{i\in I} \Hom_S(\cC_n,A)\,:\!\Psi\]
along with coordinatewise-defined chain maps and homotopies between the chain complexes of abelian groups $\bigoplus_{S_i\in I} \cC_n$ and $\bigoplus_{i\in I} C_n(S_i)$, which we continue to denote by the letters $J_\bullet$, $L_\bullet$ etc.

We note for future reference that if $\zeta\in \Hom_G(\cC_n,A)$ we have the relations
\[\Phi(\ev^*\zeta) = \prod_{i\in I} \zeta, \quad \ev^*\zeta = \Psi\left(\prod_{i\in I} \zeta\right).\]
If one abuses notation slightly to retain the notation $\zeta$ for the element of $\prod_{i\in I}\Hom_{S_i}(\cC_n,A)$ which is $\zeta$ in each coordinate, as we will do, then these relations read simply $\Phi(\ev^*\zeta)=\zeta$, $\Psi(\zeta)=\ev^*\zeta$. 

We now begin the construction. Let $(E,(q_i))$ be a relative extension of $(G,\famS)$ by $A$. Choose a set-theoretic section $r\colon G\to E$ such that $pr=\id_G$ and $r(1)=1$. Define $\zeta\in C^2(G;A)$ by the formula
\[\zeta([g_1,g_2]) = r(g_1)r(g_2)r(g_1g_2)^{-1}\]
so that $d^3\zeta=0$. We now seek a cochain $\eta\in C^1(\famS;A)$ such that $d^2\eta=\ev^*\zeta$. 

For each $i$ the data $q_i$ of the relative extension provides a cochain $\kappa_i\in \Hom_{S_i}(C_1(S_i),A)$ via the formula
\[\kappa_i([s]) = r(s)q_i(s)^{-1}.\]
One may compute that if \[\kappa=(\kappa_i)\in \prod_{i\in I} \Hom_{S_i}(C_1(S_i),A)\] then we have $d^1\kappa = I_2^*\zeta$: in each coordinate we have
\begin{eqnarray*}
d^1\kappa_i([s_1,s_2]) & = & s_1\cdot \kappa_i(s_2) + \kappa_i(s_1) - \kappa_i(s_1s_2)\\
&=& r(s_1)r(s_2)q_i(s_2)^{-1} r(s_1)^{-1}r(s_1)q_i(s_1)^{-1}q_i(s_1s_2)r(s_1s_2)^{-1}\\
& = & r(s_1)r(s_2)r(s_1s_2)^{-1}\\
& = & \zeta ([s_1,s_2])
\end{eqnarray*}

Consider now the quantity
\[J_1^*\kappa - L_1^*\zeta\in \prod_{i\in I} \Hom_{S_i}(\cC_1,A).\]
We have
\begin{eqnarray}
d^2(J_1^*\kappa - L_1^*\zeta) &=& J_2^*d^1\kappa - d^2L_1^*\zeta\nonumber\\ 
&=& J_2^*I_2^* \zeta - d^2L_1^*\zeta\nonumber\\ 
&=& \zeta\circ(I_2J_2-L_1d_2) \nonumber \\
&=& \zeta\circ(\id - d_3L_2) = \zeta \label{eq:EtaAndZeta}
\end{eqnarray}
since $d^3\zeta=0$. Applying the chain map $\Psi$ we find
\[d^2(\Psi(J_1^*\kappa - L_1^*\zeta)) = \ev^*\zeta \]
so we may set $\eta = \Psi(J_1^*\kappa - L_1^*\zeta)$.

Finally declare $\xi=\inc^*\eta$, noting that $d^2\xi=\ev^*\inc^*\zeta=0$. We associate the relative cohomology class $[\xi]\in H^2(G,\famS;A)$ to the relative extension $(E,(q_i))$.

Before we move on to discuss equivalence relations and well-definedness, we make a comment about the proof. The reader may be wondering what the slightly mysterious looking quantity $J_1^*\kappa -L_1^*\zeta$ actually means in terms of group theory. One way of constructing the chain maps $J_\bullet$ begins with an $S_i$-linear map $\Z G\to \Z S_i$ defined via a choice of a set coset representatives of $S_i$ in $G$. Having made this choice, one can construct explicit maps $J_\bullet$ and $L_\bullet$ via increasingly complex formulae involving these coset representatives. In this light, $J_1^*\kappa$ is an induction of a cochain defined over $S_i$ to one defined over $G$ using these coset representatives. The $L_1^*\zeta$ term is a correction to account for the group multiplication in $E$. When in a moment we come to discuss different choices of $J_\bullet$ and $L_\bullet$, one may think of these as showing invariance under different choices of coset representatives. The language of chain maps slims down the process by removing the need to write explicit unwieldy formulae.

We now show that our cohomology class $[\xi]$ is independent of the various choices made. Start with the set-section $r\colon G\to E$. Take a different section $\tilde r\colon G\to E$ with $\tilde r(1)=1$. Denote the cochains constructed beginning with $\tilde r$ by adding tildes, viz. $\tilde\zeta$, $\tilde \eta$ etc.

We define a cochain $\upsilon \in C^1(G;A)$ by 
\[\upsilon ([g]) = \tilde r(g)r(g)^{-1}.\]
Then we have $\tilde\zeta=\zeta+d^2\upsilon$. We also find, for each $s\in S_i$, \[\tilde\kappa_i([s])=\kappa_i([s])+\upsilon([s])\]
i.e.\ $\tilde\kappa = \kappa + I_1^*\upsilon$. We may now compute
\begin{eqnarray}
\tilde\eta-\eta&=& \Psi\left(J_1^*(\tilde\kappa-\kappa) - L_1^*(\tilde\zeta-\zeta)\right)\nonumber\\
&=& \Psi(J_1^*I_1^*\upsilon - L_1^*d^2\upsilon)\nonumber\\
&=& \Psi(\upsilon + d^1L_0^*\upsilon )\nonumber\\
&=& \ev^*\upsilon + d^1\Psi(L_0^*\upsilon). \label{eq:ChangeOfr}
\end{eqnarray}
Thus \[\tilde\xi = \xi + \inc^*\ev^*\upsilon + d^1(\inc^*\Psi(L_0^*\upsilon))\]
and $[\tilde\xi]=[\xi]$ as required.

Next consider an equivalent relative extension $(\tilde E,(\tilde q_i))$. Again use tildes to denote the new quantities derived from this extension by the above construction. We identify $\tilde E$ with $E$ via some equivalence of extensions $f$, so that $\tilde \zeta$ = $\zeta$, and each $\tilde q_i$ is $A$-conjugate to $q_i$ by some $a_i\in A$.

We find 
\begin{eqnarray*}
\tilde\kappa_i([s]) &=& r(s)\tilde q_i(s)^{-1}\\ &=& r(s)a_i^{-1}q_i(s)^{-1}a_i\\
&=& \kappa_i([s]) - s\cdot a_i + a_i
\end{eqnarray*} 
whence $\tilde\kappa = \kappa + d^1\chi$ where $\chi=(\chi_i)\in\prod_{i\in I}\Hom_{S_i}(C_0(S_i),A)$ is defined by $\chi_i([]) = -a_i$. Thus $\kappa$, and therefore $\eta$ and $\xi$, changes only by a coboundary and again $[\tilde\xi]=[\xi]$.

Finally take different choices $\widetilde J_\bullet$ and $\widetilde L_\bullet$ for our chain maps and homotopies. Take chain homotopies and 2-homotopies $M_\bullet$ and $N_\bullet$ which mediate between our choices as in Equations \eqref{eq:DefnOfM} and \eqref{eq:DefOfN}. We find
\begin{eqnarray}
\lefteqn{(\widetilde J_1^* \kappa - \widetilde L_1^*\zeta) - ( J_1^* \kappa -  L_1^*\zeta)} \qquad\qquad\nonumber\\
 &=& \kappa \circ(M_0d_1+d_2M_1) - \zeta\circ (I_2M_1-N_0d_1+d_3N_1)\nonumber\\
&=& d^1(M_0^*\kappa + N_0^*\zeta) + M_1^*(d^2\kappa-I_2^*\zeta)\nonumber \\
&=& d^1(M_0^*\kappa + N_0^*\zeta)\label{eq:ChoicesOfJandL}
\end{eqnarray}
using the relation $d^2\kappa=I_2^*\zeta$. Consequently $\tilde\xi-\xi$ is a coboundary as required.

\subsection{Conclusion}\label{sec:EquivalenceConclusion}

Finally we must establish that the two constructions from the previous subsections are actually inverse to one another. Let us deal with the more involved direction first.

Take a cocycle $\xi\in C^2(G,\famS;A)$ and construct cochains $\eta$ and $\zeta$ depending on $\xi$ to give a relative extension $(E_\zeta,(q_i))$ as in Section \ref{sec:ExtFromCochains}. Take the set theoretic section $r\colon G\to E_\zeta$ defined by $r(g)=(0,g)$ and use it to construct cochains $\tilde\zeta$, $\tilde \kappa$, etc.\ as in Section \ref{sec:CochainsFromExts}. Note that $\tilde \zeta = \zeta$. We show that $\tilde \eta - \eta$ is a coboundary, so that $[\tilde \xi]=[\xi]$ as required.

Note first that 
\begin{eqnarray*}
\tilde\kappa_i([s]) &=& r(s)q_i(s)^{-1}\\ &=& (0,s)\star (-\eta([s])(1S_i),s)^{-1}\\
&=& (\eta([s])(1S_i),1)\\ &=& \Phi(\eta)([s]),  
\end{eqnarray*}
hence $\tilde \kappa = I_1^* \Phi(\eta)$.

Now we compute
\begin{eqnarray*}
\tilde\eta &=& \Psi\left(J_1^*\tilde\kappa - L_1^*\zeta\right)\\
&=& \Psi\left(J_1^*I_1^*\Phi(\eta) - L_1^*\zeta\right)\\
&=& \Psi\left(\Phi(\eta) + d^1L_0^*\Phi(\eta) + L_1^*d^2\Phi(\eta) - L_1^*\zeta\right)\\
&=& \eta + d^1 \Psi(L_0^*\Phi(\eta))
\end{eqnarray*}
as required, noting that $\Phi(d^2\eta)=\Phi(\ev^*\zeta)=\zeta$.

Finally take a relative extension $(E,(q_i))$ and choose a set-theoretic section $r\colon G\to  E$. Produce from these data cochains $\zeta, \kappa,\ldots,\xi$ as in Section \ref{sec:CochainsFromExts}. From $\xi$ we take some choice of cochain $\tilde \eta\in C^2(\famS;A)$ such that $\inc^*\tilde\eta = \xi$; we may as well choose $\tilde \eta = \eta$. This $\eta$ then gives rise to the same cocycle $\tilde\zeta = \zeta \in C^2(G;A)$. From $\zeta$ we construct our usual extension $E_\zeta$, which is equipped with sections $\tilde q_i = (-\eta([s])(1S_i), 1)$. 

We are only left to note that $(E,(q_i))$ and $(E_\zeta, (\tilde q_i))$ are equivalent to each other via the isomorphism $f\colon E\to E_\zeta$ defined by \[f(e) = (erp(e)^{-1}, p(e)).\] We have
\begin{eqnarray*}
fq_i(s) &=& (q_i(s)r(s)^{-1}, s)\\ &=& (-\kappa_i(s),s) \\
&=& (-\eta([s])(1S_i),s)\\ &=& \tilde q_i(s)
\end{eqnarray*}
so our relative extensions are indeed equivalent.

We have shown that our constructions are inverse to one another, and have therefore established the desired relationship between cohomology and relative extensions.

\begin{theorem}\label{thm:discretecase}
Let $(G,\famS)$ be a group pair and let $A$ be a $G$-module. There is a natural bijection between $H^2(G,\famS;A)$ and the equivalence classes of relative extensions of $(G,\famS)$ by $A$.
\end{theorem}

\section{The profinite case}\label{sec:TheProfiniteCase}

\subsection{Obstructions in the profinite case}\label{sec:ProfiniteObstructions}

We would like to be able to replicate the above result for profinite group pairs. Indeed, where the family \famS\ of peripheral subgroups is finite, the proof in Section \ref{sec:RelExtsCohom} needs little modification except a liberal scattering of the word `continuous'. The problems arise when the family \famS\ is infinite. 

Recall that our family \famS\ is `continuously indexed' over a profinite set $X$. The subgroups thus depend continuously on $X$, and everything else that we write down ought to depend continuously on $X$ as well. For example, our definition of relative extension class will be modified thus:

\begin{defn}
Let $(G,\famS)$ be a profinite group pair and let $A$ be a pro-$\pi$ $G$-module. A {\em relative extension} of $(G,\famS)$ by $A$ is a pair $(E, q)$ where \[0\to A\to E \stackrel{p}{\to} G\to 1\] is a continuous extension of $G$ by $A$ and $q\colon \famS \to E$ is a continuous map such that each restriction $q_x\colon S_x\to E$ for $x\in X$ is a homomorphism such that $pq_x = \id_{S_x}$.

Two such relative extensions $(E,q)$ and $(\widetilde E, \tilde q)$ are declared to be {\em equivalent} if there is an equivalence of extensions $f\colon E \to \widetilde E$ such that there exists a continuous function $a_\bullet\colon X\to A$ witnessing that the maps $fq_x$ and $\tilde q_x$ are $A$-conjugate: 
\[\tilde q_x(s)=a_x\cdot f\circ q_x(s) \cdot a_x^{-1}  \]
for all $s\in S_x$.

The {\em trivial relative extension} is the extension $E=A\rtimes G$ with $q(s,x) = (1,s)$ for all $s\in S_x$.
\end{defn}

This definition is made for profinite $A$. For the rest of the section we restrict ourselves to finite $A$, since in general $H^2(G,\famS;A)$ is sensibly defined only for discrete torsion $G$-modules $A$.

From this starting point, we would wish to define our various cochains in a continous manner. Most of Section \ref{sec:RelExtsCohom} presents no problems in this regard: since the starting data of cocycles and pairs $(E,q)$ have the required topology, the ensuing definitions are themselves continuous. However one may expect to find difficulties where a choice is made for each $x\in X$. We would then have to ask whether such choices can be made continuously. Unfortunately, there is such a place in the proof; and what is worse, there is an obstruction to making such choices continuously.

The key issue lies in Section \ref{sec:CochainsFromExts}. From the expression $\Z{}[G/\famS] = \bigoplus_i \Z{}[G/S_i]$ we found isomorphisms 
\[\Phi\colon \Hom_G(\cC_n, \Hom(\Z{}[G/\famS],A))\longleftrightarrow\prod_{i\in I} \Hom_{S_i}(\cC_n,A)\,:\!\Psi.\]
For each $S_i$ we then exploited choices of chain equivalences of $\cC_n$ with $C_n(S_i)$ which ultimately arise from the existence of choices of coset representatives of $S_i$ in $G$. 

There are two problems with translating this to the profinite case: one serious, the other not. Firstly, we are no longer dealing with a direct sum $\bigoplus_{i\in I}$ but with a profinite direct sum $\bigboxplus_{x\in X}$. The appropriate duality will now read
\[\Phi\colon \Hom_G(\cC_n, \Hom(\Zpiof{G/\famS},A))\longleftrightarrow \Hom_{\famS }\left(\bigboxplus_{x\in X}\cC_n,A\right)\,:\!\Psi\]
where $\Hom_\famS$ denotes those homomorphisms which are $S_x$-linear on their $x$-component\footnote{One could transform $\Hom(\bigboxplus \ldots)$ into some sort of restricted direct product of $\Hom$-groups, but we would gain too little to warrant setting up even more notation.}.

It is the next step that causes a problem. There may be no continuous choice of chain map $J_\bullet$ between $\cC_n$ and $C_n(S_x)$, ultimately because there may be no way to continuously choose our coset representatives for $S_x$ in $G$. We may be unable to define chain maps $\cC_n\to C_n(S_x)$ in a continuous manner, and therefore be unable to collect these maps together into the required continuous morphisms
\[\bigboxplus_{x\in X}\cC_n\to \bigboxplus_{x\in X}C_n(S_x).\]

A simple example of this behaviour would be to take $G$ to be a cyclic group $\gp{t}$ of order 2 and for $X$ to be the one-point compactification $\N\cup\{\infty\}$ of the natural numbers, with $S_n = 1$ and $S_\infty = G$. Suppose we have the required continuous map
\[J_0\colon \bigboxplus_{x\in X} \Zpiof{G} \to\bigboxplus_{x\in X} \Zpiof{S_x} \]
i.e.\ $J_0$ restricts to an $S_x$-linear map $\Zpiof{G}\to \Zpiof{S_x}$ on each $x$-coordinate, and lifts the identity map $\Z[\pi]\to\Z[\pi]$. We use subscripts to denote elements of the different copies of $G$.

We note that there is a unique choice of $J_0$ on the $n^{\rm th}$ coordinate, for $n\in\N$, viz.\ $J_0(1_n)=J_0(t_n)=1_n$. But then by continuity we find $J_0(1_\infty)=J_0(t_\infty)=1_\infty$, contradicting the fact that the restriction to the $\infty$-coordinate should be an $S_\infty$-linear map.


We will remedy the situation by appealing to the finiteness of $A$. This will allow us, for each individual relative extension, to approximate $\famS$ by a finite object and thus circumvent the above problems.

\subsection{Finite quotients of \famS\ }
Let $G$ be a profinite group and let $\famS$ be a family of subgroups of $G$ continuously indexed over a profinite space $X$.
\begin{defn}
A {\em finite quotient} of $\famS$ is a triple $(U,\famT, Y)$ where $U\nsgp[o] G$, $\pi\colon X\to Y$ is a finite quotient space of $X$, and $\famT=\bigcup_{y\in Y} T_y\times\{y\}$ is a closed subset of $G/U\times Y$ such that:
\begin{itemize}
\item for each $y\in Y$, $T_y = \bigcup_{x\in \pi^{-1}(y)} S_xU/U$; and
\item for each $y\in Y$ there exists $x\in \pi^{-1}(y)$ such that $T_y = S_xU/U$.
\end{itemize}
\end{defn}
Note that there is a natural surjective continuous function $\pi\colon \famS \to \famT$, and that the second condition guarantees that $T_y$ is in fact a subgroup of $G/U$. Observe also that $\famT$ is somewhat redundant in the notation, being specified by $U$ and $Y$. It will occasionally be omitted from the notation.

\begin{lem}\label{lem:QuotientsOfSheaves}
Let $U\nsgp[o]G$. Then there is a clopen partition $\{A_y: y\in Y\}$ of $X$ such that for each $y\in Y$ there exists $x\in A_y$ such that $S_{x'}\subseteq S_xU$ for all $x'\in A_y$. We thus have a well-defined finite quotient $(U,\famT, Y)$ of $\famS$.
\end{lem}
\begin{proof}
Let $\cal M$ be the set of subgroups $M/U$ of $G/U$ such that $M = S_xU$ for some $x\in X$, but $M$ is not properly contained in $S_xU$ for any $x\in X$. 

Each set $X_M= \{x\in X: S_xU= M\}$ is now non-empty, and moreover is closed: if $y\notin X_M$, then there is an open neighbourhood $V$ of $y$ such that $S_v\subseteq S_yU$ for all $v\in V$, whence $V\cap X_M=\emptyset$.

Recall also that $W_M = \{x\in X: S_x\subseteq M\}$ is open in $X$. By a standard compactness argument, there we may now find disjoint clopen sets $A_M\subset X$ such that $X_M\subseteq A_M\subseteq W_M$ for each $M\in \cal M$.  

These $A_M$ form part of our desired partition of $X$. We now repeat the process, replacing $X$ with $X\smallsetminus \bigcup A_M$. Note that each member of our new family of maximal subgroups $\cal M$ will be a proper subgroup of a member of the old family. This process will thus terminate after finitely many steps, since $G/U$ is finite.  
\end{proof}
\begin{prop}\label{prop:FactoringThruQuotients}
Let $f\colon \famS\to F$ be a continuous function to a finite set. Then there is a finite quotient $(U,\famT, Y)$ of $\famS$ such that $f$ factors through the quotient map $\famS\to\famT$.
\end{prop}
\begin{proof}
Consider the product map $g=f\times f\colon \famS\times\famS \to F\times F$, and the diagonal subset $D\subseteq F\times F$. For each $r=((s,x),(s',x'))\in g^{-1}(D)$, there are clopen subsets $B_r$ and $C_r$ of $X$ and open normal subgroups $U_r$ and $V_r$ of $G$ such that 
\[((s,x),(s',x'))\in \left((sU_r\times B_r)\cap\famS\right) \times \left((s'V_r\times C_r)\cap\famS\right) \subseteq g^{-1}(D).\]
We may cover the compact space $g^{-1}(D)$ by a finite collection of such `boxes'. Intersecting all the $U_r$ and $V_r$, and producing a partition from intersections of all the $B_r$ and $C_r$, we obtain an open normal subgroup $U$ of $G$ and a partition $\{A_y:y\in Y\}$ of $X$ such that whenever $f(s,x)=f(s',x')$ we have
\[((s,x),(s',x'))\in \left((sU\times A_y)\cap\famS\right) \times \left((s'U\times A_{y'})\cap\famS\right) \subseteq g^{-1}(D)\]
where $x\in A_y$, $x'\in A_{y'}$. Then $f$ factors through the natural map to the finite space $\famT \subseteq G/U \times Y$ defined by $T_y = \bigcup_{x\in \pi^{-1}(y)} S_xU/U$.

It may not yet be the case that each $T_y$ is actually a subgroup of $G/U$. For each $y\in Y$ we then apply Lemma \ref{lem:QuotientsOfSheaves} to refine our partition further and obtain a genuine finite quotient $(U,\famT', Y')$ through which $f$ factors.
\end{proof}

We define a partial order on the set of finite quotients by $(U,\famT,Y)\succeq (U',\famT',Y')$ if $U\leq U'$ and $X\to Y'$ factors through $Y$. One consequence of the above proposition is to give us a representation 
\[\famS = \varprojlim_{(U,\famT, Y)} \famT \]
of \famS\ as an inverse limit of its finite quotients.

\subsection{Equivalence of cohomology and extensions}\label{sec:ProfiniteRelExts}

We are now ready to demonstrate the profinite counterpart of Theorem \ref{thm:discretecase}. A large amount of this proof is an almost verbatim copy of Section \ref{sec:RelExtsCohom} and we shall not repeat those portions of the argument.

Let $(G,\famS)$ be a profinite group pair, where $\famS$ is continuously indexed over the profinite set $X$. Let $\pi$ be a non-empty set of primes and let $A$ be a $\pi$-primary finite $G$-module.

To begin with, we note that Section \ref{sec:ExtFromCochains} requires no modifications beyond adding the word `continuous' in myriad places. None of the issues discussed in Section \ref{sec:ProfiniteObstructions} present themselves, and we shall move on.

We must now construct the inverse function, which takes a class of relative extensions and produces a cohomology class. We will accomplish this by expressing the set of extensions as a direct limit.

Let $\Ext$ be the set of triples $(E,q,r)$ where $(E,q)$ is a relative extension of $(G,\famS)$ by the finite $G$-module $A$, and $r\colon G\to E$ is a continuous set theoretic section of the map $E\to G$ with $r(1)=1$. Given $r$ we produce the continuous functions 
\begin{eqnarray*}\zeta\colon G^2 \to A,&& \zeta([g_1,g_2]) = r(g_1)r(g_2)r(g_1g_2)^{-1}\\
\kappa\colon \famS\to A,&& \kappa(s,x) = r(s)q_x(s)^{-1}.\end{eqnarray*}
For a finite quotient $(U,\famT,Y)$ of \famS, let $\Ext_{U,Y}$ denote the set of triples $(E,q,r)$ such that $\zeta$ factors through a map $\zeta^U\colon G/U\to A$ and $\kappa$ factors through a map $\kappa^U\colon \famT\to A$. Observe that $\zeta^U$ and $\kappa^U$ are uniquely defined.

By Proposition \ref{prop:FactoringThruQuotients} we find that for any $(E,q,r)$ there exists some finite quotient through which $\zeta$ and $\kappa$ factor. We thus conclude that $\Ext$ is a direct limit
\[\Ext = \varinjlim_{(U,Y)} \Ext_{U,Y}.\]

Take $(E,q,r)\in \Ext_{U,Y}$. We will use the symbol $\pi_U$ to denote the quotient maps $\famS\to \famT$, $G\to G/U$ and the various chain maps consequent on these quotients.

For any profinite group $H$ we will use $C_\bullet(H)$ to denote the standard resolution of $\Z[\pi]$ by \Zpiof{H}-modules; that is,
\[C_n(H) = \Zpiof{H}[\![H^{(n)}]\!].\]
We then have (cf Equation \eqref{eq:MapsOfDirectSums}) a commuting diagram of chain complexes 
\[\begin{tikzcd}
\bigboxplus_{x\in X} C_\bullet(S_x) \ar{r}{\pi_U}\ar{d}{I_\bullet} & \bigoplus_{y\in Y}C_\bullet(T_y)\ar{d}{I_\bullet}\\
\bigboxplus_{x\in X} C_\bullet(G) \ar{r}{\pi_U}& \bigoplus_{y\in Y}C_\bullet(G/U)
\end{tikzcd}\]
where as before $I_\bullet$ denotes maps induced by inclusions of groups. Since the maps $G\to G/U$ and $\famS \to \famT$ are surjective, we inherit the relations
\[d^3\zeta^U = 0,\quad d^2\kappa^U = I_2^\ast\zeta^U \]
from Section \ref{sec:ExtFromCochains}.

For each $y$, the complex $C_\bullet(G/U)$ is a projective resolution of \Z[\pi] by $T_y$-modules. Applying Section \ref{sec:ChainHomotopies} on each coordinate we obtain a chain map
\[J_\bullet\colon  \bigoplus_{y\in Y} C_\bullet(G/U) \to  \bigoplus_{y\in Y} C_\bullet(T_y)\]
lifting the identity map on \Z[\pi], and a chain homotopy
\[L_\bullet\colon \bigoplus_{y\in Y} C_\bullet(G/U) \to \bigoplus_{y\in Y} C_{\bullet+1}(G/U) \]
satisfying Equation \eqref{eq:DefnOfL}. Each of these maps is $T_y$-linear on the $y$-coordinate.

We may now define a cochain
\[\eta = \Psi\left(\pi_U^\ast(J_1^\ast\kappa^U - L_1^\ast \zeta^U)\right)\in C^1(\famS;A).\]  
\begin{figure}[ht]
\centering

\[\begin{tikzcd}
 & & \bigoplus_{y\in Y} C_1(T_y) \ar{dr}{\kappa^U} & \\
 \bigboxplus_{x\in X} C_1(G) \ar{r}{\pi_U}& \bigoplus_{y\in Y}C_1(G/U) \ar{ur}{J_1} \ar{dr}{L_1} & & A\\
  & & \bigoplus_{y\in Y} C_2(G/U) \ar{ur}{\zeta^U} & 
\end{tikzcd}\]\caption{Construction of the cochain $\eta$}
\end{figure}
We note that 
\[
d^2\eta = \Psi\left( \pi_U^\ast d^2(J_1^\ast\kappa^U - L_1^\ast \zeta^U)\right) = \Psi(\pi_U^\ast \zeta^U)  = \ev^\ast \zeta
\]
using similar manipulations to Equation \eqref{eq:EtaAndZeta}. We are thus entitled to mimic Section \ref{sec:CochainsFromExts} and define a function 
\[\epsilon_{U,Y}\colon \Ext_{U,Y} \to H^2(G,\famS;A), \quad (E,q,r)\mapsto [\inc^\ast\eta].\]
Note that this function is well-defined. To define $\eta$ we made choices of $J_\bullet$ and $L_\bullet$. Different choices of $J_\bullet$ and $L_\bullet$ are mediated by chain homotopies and 2-homotopies $M_\bullet$ and $N_\bullet$ defined on $C_\bullet(G/U)$ as in Section \ref{sec:CochainsFromExts}, so $\inc^\ast\eta$ changes only by a coboundary, via computations similar to Equation \eqref{eq:ChoicesOfJandL}.  

We must now check that these functions $\epsilon_{U,Y}$ are consistent with the inclusion maps of our direct limit $\Ext = \varinjlim \Ext_{U,Y}$. Take some $(E,q,r)\in \Ext_{U,Y}$ and let $(V,{\cal W},Z)$ be a finite quotient of $\famS$ with $(V,Z)\succeq (U,Y)$. Note that $\zeta$ and $\kappa$ both factor over $(V,Z)$. Regarding $(E,q,r)$ as an element of both $\Ext_{U,Y}$ and of $\Ext_{V,Z}$, produce cochains $\eta^U$ and $\eta^V$ via the above process, decorating the symbols $\zeta, \kappa, J, L,\ldots$ with superscripts $U$ and $V$ appropriately. We use $\rho$ to denote the transition maps ${\cal W}\to \famT$, $Z\to Y$ and the chain maps consequent on these.

We now have a commuting diagram of chain complexes:
\[\begin{tikzcd}
\bigboxplus_{x\in X} C_\bullet(S_x) \ar{r}{\pi_V}\ar{d}{I_\bullet} & \bigoplus_{z\in Z}C_\bullet(W_z)\ar{r}{\rho}\ar{d}{I_\bullet}& \bigoplus_{y\in Y}C_\bullet(T_y)\ar{d}{I_\bullet}\\
\bigboxplus_{x\in X} C_\bullet(G) \ar{r}{\pi_V}& \bigoplus_{z\in Z}C_\bullet(G/V)\ar{r}{\rho}& \bigoplus_{y\in Y}C_\bullet(G/U)
\end{tikzcd}\]
For each $z\in Z$ we have a pair of maps of chain complexes of $W_z$-modules\footnote{We have $\rho(W_z)\subseteq T_{\rho(z)}$, so we may regard $C_\bullet(T_{\rho(z)})$ as a $W_z$-module.}
\[\rho J^V_\bullet, J^U_\bullet\rho\colon C_\bullet(G/V)\to C_\bullet (T_{\rho(z)}) \]
each of which lifts the identity map on \Z[\pi]. 

Since $C_\bullet(G/V)$ is projective as a $W_z$-module, from Section \ref{sec:ChainHomotopies} we have a chain homotopy
\[M_\bullet \colon C_\bullet(G/V) \to C_{\bullet+1}(T_{\rho(z)})\]
satisfying
\[\rho J^V_n - J^U_n \rho = M_{n-1}d_n + d_{n+1}M_n.\]

We similarly have chain homotopies $\rho L^V_\bullet$ and $L^U_\bullet \rho$ mediating between the maps $I_\bullet \rho J^V_\bullet$ and $I_\bullet J^U_\bullet \rho$ and the identity respectively. Then from Section \ref{sec:ChainHomotopies} we find a 2-homotopy 
\[ N_\bullet \colon C_\bullet(G/V) \to C_{\bullet+2} (G/U)\]
such that 
\[ \rho L^V_n - L^U_n \rho = I_{n+1}M_n +  d_{n+2}N_n - N_{n-1} d_n.\]

Using the relations $\pi_U = \rho \pi_V$, $\kappa^V = \kappa^U \rho$ and so on, we may now compute
\begin{eqnarray*}
\lefteqn{\pi_V^\ast( J_1^{V *} \kappa^V -  L_1^{V *}\zeta^V) - \pi_U^\ast( J_1^{U *} \kappa^U -  L_1^{U *}\zeta^U)} \qquad\qquad\nonumber\\
 &=& \pi_V^\ast\left( (\kappa^U\circ(\rho J_1^V) - \zeta^U\circ(\rho L_1^V)) - (\kappa^U\circ( J_1^U\rho) - \zeta^U\circ( L_1^U \rho))  \right)\\
 &=&  \pi_V^\ast\left( \kappa^U \circ(M_0d_1+d_2M_1) - \zeta^U\circ (I_2M_1-N_0d_1+d_3N_1)\right)\nonumber\\
&=& \pi_V^\ast\left(d^1(M_0^*\kappa^U + N_0^*\zeta^U) + M_1^*(d^2\kappa^U-I_2^*\zeta^U)\right) \nonumber \\
&=& d^1\pi_V^\ast\left(M_0^*\kappa^U + N_0^*\zeta^U\right)
\end{eqnarray*}
whence $\eta^U$ and $\eta^V$ differ only by a coboundary. 

We next pass from a map $\Ext\to H^2(G,\famS;A)$ to a map from the set of equivalence classes of relative extensions by imposing two equivalence relations: the relation that forgets the choice of section $r\colon G\to E$, and the relation of equivalence of relative extensions.

Firstly, the choice of $r$. Let $r$ and $\tilde r$ be two choices of section $G\to E$ and use tildes to distinguish the various definitions made depending on these choices. Define the continuous cochain $\upsilon\in C^1(G;A)$ by 
\[\upsilon([g]) = \tilde r(g) r(g)^{-1}.\]  
Now take some finite quotient $(U,\famT,Y)$ of $\famS$ such that all the maps $\kappa$, $\tilde \kappa$, $\zeta$, $\tilde \zeta$ and $\upsilon$ factor over $(U,Y)$. We have, as before, relations $\tilde \kappa = \kappa + I_1^\ast \upsilon$ and $\tilde\zeta = \zeta + d^2\upsilon$. Computations similar to Equation \eqref{eq:ChangeOfr} now yield
\[\tilde\eta - \eta = \ev^\ast\upsilon + d^1\Psi(\pi_U^\ast L_0^\ast\upsilon^U).\]
so that $[\inc^\ast\tilde \eta]= [\inc^\ast\eta]$ as required.

Next consider an equivalent extension $(\widetilde E,\tilde q)$ to $(E,q)$. We identify $E$ with $\widetilde E$ by some equivalence of extensions $f$, so that $\tilde q$ is $A$-conjugate to  $q$ via some $a_\bullet\colon X\to A$. As previously, we have $\tilde \zeta = \zeta$, and $\tilde \kappa = \kappa + d^1\chi$ where $\chi\colon \bigboxplus_{x\in X} C_0(S_x)\to A$ is defined by $\chi_x([]) = - a_x$.

We may choose some $r\colon G\to E$, and take some $(U,\famT,Y)$ over which the cochains $\zeta$, $\kappa$ and $a_\bullet$ all factor. Then $\chi$ factors over a map $\chi^U\colon \bigoplus_{y\in Y} C_0(T_y)\to A$, and the equation $\tilde \kappa^U= \kappa^U + d^1\chi^U$ still holds. Once again, $\tilde\eta$ and $\eta$ differ only by a coboundary.

We ought finally to reprise Section \ref{sec:EquivalenceConclusion}, to show that our functions 
\[H^2(G,\famS;A)\longleftrightarrow \{\text{Relative extension classes}\} \]
are indeed inverses to one another. No reader who has successfully swum through the torrent of notation in the last few pages will find any surprises here. There is no material difference from Section \ref{sec:EquivalenceConclusion}, save that in order to define certain things we must pass through a suitably chosen quotient $(U,\famT,Y)$. We will show mercy by omitting further details.
\begin{theorem}\label{thm:profinitecase}
Let $(G,\famS)$ be a profinite group pair and let $A$ be a finite $G$-module. There is a natural bijection between $H^2(G,\famS;A)$ and the equivalence classes of relative extensions of $(G,\famS)$ by $A$.
\end{theorem}

\section{Projective profinite group pairs}\label{sec:Projective}

Let $\cal C$ be a variety of finite groups, and let $\pi(\cC)$ be the set of primes involved in \cC. Let $G$ be a pro-$\cC$ group with a family of subgroups \famS\ continuously indexed over a profinite set $X$.

\begin{defn}
A {\em pro-\cC\ lifting problem of $G$ relative to \famS, with kernel $K$}, consists of the following data:
\begin{itemize}
\item a surjection of pro-\cC\ groups $\alpha\colon H\to Q$ with kernel $K$;
\item a homomorphism $\phi\colon G\to Q$; and
\item a continuous function $\sigma\colon\famS\to H$ such that each $\sigma|_{S_x}$ is a homomorphism with $\alpha\circ \sigma|_{S_x} = \phi|_{S_x}$.\end{itemize} 
A {\em solution} of this lifting problem\footnote{This is analogous to what is called a `weak solution' in \cite[Section 3.5]{RZ00}.} consists of a homomorphism $\bar \phi\colon G\to G$ such that $\alpha\circ \bar\phi = \phi$, and a continuous function $c\colon X\to K$ such that 
\[c(x)\bar\phi(s)c(x)^{-1} = \sigma(s,x)\quad \forall s\in S_x.\]

The lifting problem is {\em finite} if $H$ is finite.
\end{defn} 
\begin{defn}[cf {\cite[Section A.4]{Haran07}}]
The pair $(G,\famS)$ is a {\em projective pro-\cC\ pair} if every finite pro-\cC\ lifting problem of $G$ relative to \famS\ has a solution.
\end{defn}

First let us examine the case where $K=A$ is a finite abelian group, and thus a $Q$-module. We may consider $A$ to be a $G$-module via the map $\phi$. We may form the pullback \[E=\{(g,h)\in G\times H \mid \phi(g)=\alpha(h)\},\] an extension of $G$ by $A$. The continuous function $q\colon \famS\to E$, $q(s,x)= (s,\sigma(s,x))$ gives us a relative extension $(E,q)$. 

If our lifting problem has a solution $(\bar\phi,c)$, then we have an equivalence of extensions
\[f\colon E\to A\rtimes G,\quad f(g,h) = (h\bar\phi(g)^{-1}, g)\]
which satisfies 
\[f\circ q(s,x) = (\sigma(s,x)\bar\phi(s)^{-1},s) = (c(x),1)\star (1,s)\star (c(x),1)^{-1} \]
so that $(E,q)$ is equivalent to the trivial relative extension. Similarly, the converse holds: if the pullback is equivalent to the trivial relative extension then the lifting problem has a solution. 
\begin{prop}
A lifting problem with finite abelian kernel $A$ has a solution if and only if the associated pullback relative extension represents the zero element of $H^2(G,\famS;A)$.
\end{prop} 
\begin{clly}
The following are equivalent:
\begin{itemize}
\item Every pro-\cC\ lifting problem of $G$ relative to \famS\ with finite abelian kernel has a solution.
\item $\cd_p(G,\famS)\leq 1$ for every $p\in \pi(\cC)$.
\end{itemize}
\end{clly}

We now have three related concepts.
\begin{enumerate}[(I)]
\item $\cd_p(G,\famS)\leq 1$ for every $p\in \pi(\cC)$.\label{item:cd}
\item $(G,\famS)$ is a projective pro-\cC\ pair.\label{item:projective}
\item $G= \coprod_{x\in X} S_x \amalg F$ is a free pro-\cC\ product of the family \famS\ with a projective pro-\cC\ group $F$.\label{item:freeprod}
\end{enumerate}

It is clear that (\ref{item:freeprod}) implies (\ref{item:projective}), and we have just shown that (\ref{item:projective}) implies (\ref{item:cd}). When the family \famS\ is empty, all three notions are equivalent \cite[Section 7.6]{RZ00}.

We will not, in general, have (\ref{item:projective})$\Rightarrow$(\ref{item:freeprod}), for the usual reason that closed subgroups of free profinite products need not themselves be free profinite products. We also note that (\ref{item:cd}) need not imply (\ref{item:projective}).
\begin{example}
Let $\cC_{3,5}$ be the set of finite groups whose orders are divisible only by the primes $3$ and $5$, and similarly define $\cC_{2,3,5}$. Let $G$ be the free pro-$\cC_{3,5}$ product $\Z[3]\amalg_{3,5} \Z[5]$, and take $\famS=\{\Z[3],\Z[5]\}$. 

Certainly $(G,\famS)$ is a projective pro-$\cC_{3,5}$ pair, whence $\cd_p(G,\famS)\leq 1$ for $p=3,5$. Since $G$ is a pro-$\cC_{3,5}$ group, we also have $\cd_2(G,\famS)=0$. 

However $(G,\famS)$ is not a projective pro-$\cC_{2,3,5}$ pair: the lifting problem with $H={\rm Sym}_5$, $Q=1$, $\sigma(\Z[3]) = \gp{(1\, 2\, 3)}$, and $\sigma(\Z[5]) = \gp{(1\, 2\, 3\,4\, 5)}$ has no solution since ${\rm Sym}_5$ has no subgroup of order 15.
\end{example}

To restore the equivalence of (\ref{item:cd}) and (\ref{item:projective}), we will repair to the world of prosolvable groups.
\begin{prop}
Let $G$ be a pro-$\cal C$ group, where $\cal C$ is a family of solvable groups. The following statements are equivalent.
\begin{enumerate}[(1)]
\item Every pro-$\cal C$ lifting problem of $G$ relative to $\famS$ with kernel a finite abelian $\cC$-group has a solution.\label{item:fabkernel}
\item Every pro-$\cal C$ lifting problem of $G$ relative to $\famS$ with finite kernel has a solution. In particular $(G,\famS)$ is a projective pro-$\cal C$ pair.\label{item:fkernel}
\item Every pro-$\cal C$ lifting problem of $G$ relative to $\famS$ has a solution.\label{item:anykernel}
\end{enumerate}
\end{prop}
\begin{proof}
The implications $(3) \Rightarrow (2) \Rightarrow (1)$ are immediate. We first establish that (\ref{item:fabkernel}) implies (\ref{item:fkernel}) by induction on the order of the finite solvable group $K$.

If $K$ is abelian, we have nothing to do. Otherwise, we have the proper subgroup $K'=[K,K]$ of $K$ which is normal in $H$. Denote the quotient maps by $\alpha'\colon H/K' \to Q$ and $\beta\colon H\to H/K'$. We now have a lifting problem with data $(\alpha',\phi, \sigma')$ where  $\sigma'=\beta\sigma\colon \famS\to H/K'$. This lifting problem, which has kernel $K/K'$, has by induction a solution $\phi'\colon G\to H/K'$, $c\colon X\to K/K'$, so that 
\[\beta\sigma(s,x) = c(x)\phi'(s)c(x)^{-1}.\]
Now choose some section $r\colon K/K'\to K$ and consider the lifting problem with data $(\beta, \phi',\bar\sigma)$ where \[\bar \sigma(s,x)= rc(x)^{-1}\sigma(s,x) rc(x).\] By induction, this lifting problem has a solution $\bar\phi\colon G\to H$, $\bar c\colon X\to K$. This provides a solution to the original lifting problem $H\to Q$.

%
%
%

Now we approach the implication $(\ref{item:fkernel})\Rightarrow (\ref{item:anykernel})$ by a d\'evissage procedure. Take a pro-$\cal C$ lifting problem of $G$ relative to \famS.

Let $\cal P$ be the set of triples $(L,\eta,c)$ where $L\leq K$ is a normal subgroup of $H$ and ($\eta\colon G\to H/L$, $c\colon X\to K/L$) is a solution of the lifting problem with data $(\alpha_L, \phi, \beta_L\sigma)$ where $\alpha_L$ and $\beta_L\colon H\to H/L$ are the natural maps induced by the quotient. That is, $\alpha_L \eta = \phi$ and \[\beta_L\sigma(s,x)=c(x) \eta(s) c(x)^{-1}.\] 

If $L_2\leq L_1$ then let $\gamma\colon H/L_2\to H/L_1$ be the quotient map. We write $(L_2,\eta_2,c_2)\succeq (L_1,\eta_1,c_1)$ if $L_2\leq L_1$, $\gamma\eta_2=\eta_1$ and $\gamma c_2= c_1$. Note that $\cal P$ is a non-empty poset since $(K, \phi,1)\in\cal P$. Also, $\cal P$ is an inductive poset: for any chain $((L_n, \eta_n, c_n))_{n\in I}$ in $\cal P$ we have $(\bigcap L_n, \varprojlim \eta_n, \varprojlim c_n)\in \cal P$.

Therefore $\cal P$ has a maximal element $(L,\eta,c)$. If $L=1$, we have solved our lifting problem. Let us suppose $L\neq 1$ and derive a contradiction. If $L\neq 1$ there is an open normal subgroup $H'$ of $H$ such that $L'=H'\cap L\neq L$. Again let $\gamma\colon H/L'\to H/L$ denote the quotient. Choose some continuous section $r\colon K/L \to K/L'$.

We now have a lifting problem with data $(\gamma, \eta, \sigma')$ where $\gamma\colon H/L'\to H/L$, $\eta\colon G\to H/L$, and $\sigma'\colon \famS\to H/L'$ is defined by
\[\sigma'(s,x) = rc(x)^{-1} \beta_{L'}\sigma(s,x) rc(x).\]
Note that $\gamma\sigma'(s,x) = \eta(s)$ as required.

This lifting problem has finite kernel $L/L'$, and thus has a solution ($\eta'\colon G\to H/L'$, $c'\colon X\to L/L'$). Then $\alpha_{L'}\eta' = \phi$ and \[rc(x)c'(x) \eta'(s) (rc(x)c'(x))^{-1} = \beta_{L'}\sigma(s,x),\]
so $(L',\eta', rc \cdot c') \in \cal P$ and $(L',\eta', rc \cdot c')\succeq (L,\eta,c)$, contradicting the assumed maximality. The proof is complete.
\end{proof}

This establishes the equivalence (\ref{item:cd}) $\Leftrightarrow$ (\ref{item:projective}) for prosolvable groups. For pro-$p$ groups we can go slightly further and find projectivity in the category of all profinite groups.

\begin{prop}
Let $G$ be a pro-$p$ group. The following statements are equivalent.
\begin{enumerate}[(1)]
\item Every profinite lifting problem of $G$ relative to $\famS$ with kernel a finite abelian $p$-group has a solution.\label{item:fabkernel2}
\item Every profinite lifting problem of $G$ relative to $\famS$ with kernel a finite $p$-group has a solution.\label{item:fpkernel2}
\item Every profinite lifting problem of $G$ relative to $\famS$ with finite kernel has a solution. In particular $(G,\famS)$ is a projective profinite pair.\label{item:fkernel2}
\item Every profinite lifting problem of $G$ relative to $\famS$ has a solution. \label{item:anykernel2}
\end{enumerate}
\end{prop}
\begin{proof}
The implications $(4)\Rightarrow (3) \Rightarrow (2) \Rightarrow (1)$ are immediate. The implications  (\ref{item:fabkernel2}) $\Rightarrow$ (\ref{item:fpkernel2}) and $(\ref{item:fkernel2})\Rightarrow (\ref{item:anykernel2})$ are identical to the analogous implications in the last proposition.

It remains to note the implication (\ref{item:fpkernel2}) $\Rightarrow$ (\ref{item:fkernel2}). 

Take some lifting problem $(\alpha\colon H\to Q$, $\phi\colon G\to Q$, $\sigma\colon \famS\to H)$ with finite kernel $K$. Let $Q'$ be a pro-$p$ Sylow subgroup of $H$ which contains $\phi(G)$, and take $H'=\alpha^{-1}(P)$. Fix one pro-$p$ Sylow subgroup $P$ of $H'$ and note that $\alpha(P) = Q'$.

Each pro-$p$ subgroup $\sigma(S_x)$ is contained in one of the finitely many pro-$p$ Sylow subgroups of $H'$, and thus is conjugate by $K$ into $P$. By a standard compactness argument we may find a continuous function $c\colon X\to K$ such that $c(x)\sigma(S_x)c(x)^{-1}\leq P$ for all $x\in X$.

We have now replaced our original lifting problem with a lifting problem $P\to Q'$ solely in the category of pro-$p$ groups, replacing $\sigma$ with its conjugate by $c$. By assumption this lifting problem has a solution, which immediately gives a solution of our original lifting problem.
\end{proof}

We conclude with the equivalence (\ref{item:projective}) $\Leftrightarrow$ (\ref{item:freeprod}) for pro-$p$ groups.

\begin{prop}
Let $G$ be a pro-$p$ group such that $\cd_p(G,\famS)\leq 1$. Then there exists a free pro-$p$ group $F$ such that $G$ is a free pro-$p$ product
\[G= \coprod_{x\in X}S_x \amalg F.\]
\end{prop}
\begin{proof}
In this proof all homology and cohomology groups have $\F_p$ coordinates. 

The discrete elementary abelian $p$-group $\ker (H^1(G) \to H^1(\famS))$ is isomorphic to $\bigoplus_I \F_p$ for some indexing set $I$ \cite[Remark 7.7.1]{RZ00}. Let $F$ be the free pro-$p$ group on the set $I$ converging to 1. Since $H^1(G)$ is dual to $H_1(G)$, we have $H_1(G)/\im H_1(\famS) \iso H_1(F)$. Note that $H_2(G,\famS)=0$ so $H_1(\famS)$ injects into $H_1(G)$. Note also that $H_1(\famS; \F_p) = \bigboxplus_{x\in X} H_1(S_x; \F_p)$ from, for example, \cite[Theorem 9.4.3(b)]{RZ00}.

Since $F$ is free we may lift $F\to H_1(F)$ to a map $F\to G$ such that the natural map $f\colon \coprod S_x \amalg F \to G$ is an isomorphism on $H_1$, and therefore a surjection. Take a solution $\bar\phi\colon G\to \coprod S_x \amalg F$ of the lifting problem defined by $\id\colon G\to G$. Since $\bar\phi$ is also an isomorphism on $H_1$, it is a surjection; therefore both $f$ and $\bar\phi$ are isomorphisms as required.
\end{proof}
\begin{theorem}\label{thm:cd1isproj} Let $(G,\famS)$ be a pro-$p$ group pair, where $\famS=\{S_x\}_{x\in X}$ is a family of subgroups continuously indexed over the profinite set $X$. The following are equivalent.
\begin{enumerate}[(I)]
\item $\cd_p(G,\famS)\leq 1$.
\item $(G,\famS)$ is a projective profinite pair.
\item $G= \coprod_{x\in X} S_x \amalg F$ is a free profinite product of the family \famS\ with a free pro-$p$ group $F$.
\end{enumerate}
\end{theorem}

Since cohomological dimension is inherited by closed subgroups, this theorem (or more precisely the equivalence (I) $\Leftrightarrow$ (III)) provides a new proof of the Kurosh Subgroup Theorem for pro-$p$ groups, which avoids the use of profinite trees. This theorem is originally due to Melnikov \cite{Melnikov90} and Haran \cite{Haran87} independently.

\begin{theorem}[Kurosh Subgroup Theorem]
Let $G= \coprod_{x\in X} S_x \amalg F$ for some free pro-$p$ group $F$, where $\famS = \{S_x\colon x\in X\}$ is a continuously indexed family of subgroups of $G$. 

Let $H\leq G$ be a closed subgroup. Suppose there is a continuous section $r\colon H\lqt G/\famS \to G/\famS$ of the quotient map $G/\famS \to H\lqt G/\famS$ and let $\famS^H$ be the family of subgroups 
\[\famS^H = \{H\cap r(y)S_xr(y)^{-1} \mid x\in X, y\in H\lqt G/S_x\}\]
continuously indexed over the profinite set 
\[H\lqt G/\famS = \coprod_{x\in X} H\lqt G/S_x.\]
Then there is a free pro-$p$ group $F'$ such that 
\[H = \coprod_{H\lqt G/\famS} \left(H\cap r(y)S_xr(y)^{-1}\right) \amalg F'. \]
\end{theorem}
\begin{rmk}
The usual statement of the theorem includes a condition that $H$ is second-countable; the purpose of this condition is to force the existence of the section $r$. See \cite[Theorem 9.6.1(b) and Theorem 9.6.2(a)]{Ribes17}.
\end{rmk}

\bibliographystyle{alpha}
\bibliography{RelCoh.bib} 
\end{document}